\documentclass[12pt]{amsart}
\usepackage{imakeidx}
\makeindex[columns=3]
\usepackage[utf8]{inputenc}
\usepackage{tgtermes}
\usepackage[english]{babel}
\usepackage{amsmath,
	mathtools,
	amssymb,
	graphicx,
	bbm,
	xcolor,
	amsthm,
	mathrsfs,
	enumitem,
	mathdots}
\usepackage[all]{xy}
\usepackage{tikz}
\usepackage{listings}
\theoremstyle{plain}

\newtheorem{theorem}{Theorem}[section]

\newtheorem{lemma}[theorem]{Lemma}
\newtheorem{proposition}[theorem]{Proposition}

\newtheorem{remark}[theorem]{Remark}

\newtheorem*{theorem-no-num}{Theorem}
\newtheorem*{proposition-no-num}{Proposition}
\newtheorem*{corollary-no-num}{Corollary}

\DeclareFontFamily{U}{wncy}{}
\DeclareFontShape{U}{wncy}{m}{n}{<->wncyr10}{}
\DeclareSymbolFont{mcy}{U}{wncy}{m}{n}
\DeclareMathSymbol{\Sha}{\mathord}{mcy}{"58}

\usepackage{hyperref}
\begin{document}
	\title{On the convergence of certain indefinite theta series}
	\author{Xiaoyu Zhang}
	\address{Universit\"{a}t Duisburg-Essen,
		Fakult\"{a}t f\"{u}r Mathematik,
		Mathematikcarr\'{e}e
		Thea-Leymann-Straße 9,
		45127 Essen,
		Germany}
	\email{xiaoyu.zhang@uni-due.de}
	\maketitle

\begin{abstract}
	We give an elementary proof of the convergence of indefinite theta series associated to an inner space of signature $(n,2)$ conjectured in the work of Alexandrov,Banerjee,Manschot and Pioline (\cite[C]{AlexandrovBanerjeeManschotPioline2018b}) and show that the incidence conditions in \textit{loc.cit} are also necessary for the convergence.
\end{abstract}

\tableofcontents

\section{Introduction}
Theta series are one of the most fundamental examples in the theory of modular forms. The theta series associated to a lattice $(L,Q)$ with integer valued positive definite quadratic form $Q$ is classically known to converge absolutely termwise and define a holomorphic modular form:
\[
\theta(\tau)
=
\sum_{x\in L}
q^{Q(x)},
\quad
q=
e^{-2i\pi\tau}
\text{ and }
\tau=u+vi\in\mathbb{C}
\text{ with }
v>0.
\]
Write $(-,-)$ for the symmetric bilinear form associated to $Q$ such that $Q(x)=\frac{1}{2}(x,x)$. When the quadratic space is indefinite, the theta series is no longer absolutely convergent (because the isometry group of $(L,Q)$ is infinite) and there are several different approaches to take care of the convergence problem. One approach is to sum over a subset of $L$. In this case the theta series is holomorphic but generally no longer a modular form. However, in some situations the series can be completed to a non-holomorphic modular form. This direction of research began with the work of G\"{o}ttsche and Zagier \cite{GottscheZagier1998} and Zwegers (\cite{Zwegers2002}) for $V$ of signature $(n,1)$: fix two negative vectors $C_1,C_2$ in $V$ (that is, $(C_1,C_1),(C_2,C_2)$ are negative) such that $(C_1,C_2)<0$ and take $\mu\in L^\vee$ (the dual of $L$), then put
\[
\theta_{\mu}(\tau;\{C_1,C_2\})
=
\sum_{x\in\mu+L}
\Phi(x;\{C_1,C_2\})q^{(x,x)/2}
\]
where $\Phi(x;\{C_1,C_2\})
=
\mathrm{sgn}((x,C_1))-\mathrm{sgn}((x,C_2))$.  Then Zwegers shows that this theta series is termwise absolutely convergent and that $\theta_{\mu}$ be completed to a modular form using certain complementary error function (we refer the reader to\cite{Zagier2010} for an overview).
In \cite{AlexandrovBanerjeeManschotPioline2018a}, Alexandrov, Banerjee, Manschot and Pioline established an analogue of $\Phi(x;\{C_1,C_2\})$ for $V$ of signature $(n,2)$ and later on, in \cite{AlexandrovBanerjeeManschotPioline2018b}, based on motivations from physics, they consider a variant of $\Phi(x;\{C_1,C_2\})$ for $V$ of signature $(n,2)$ and conjecture the absolute convergence of the corresponding theta series. This is what we will consider in this note.

More precisely, take $N$ non-zero vectors $\mathcal{C}=\{C_1,\cdots,C_N\}$ in $V$ and consider the following incidence conditions on these vectors:
\begin{align}
	\tag{I.1}
	(C_j,C_j)
	\le0,
	\\
	\tag{I.2}
	(C_j,C_j)(C_{j+1},C_{j+1})-(C_j,C_{j+1})^2
	\begin{cases*}
		>0,
		&
		if
		$(C_j,C_j)(C_{j+1},C_{j+1})>0$,
		\\
		=0,
		&
		if
		$(C_j,C_j)(C_{j+1},C_{j+1})=0$,
	\end{cases*}
	\\
	\tag{I.3}
	\begin{cases*}
		(C_j,C_j)(C_{j-1},C_{j+1})-(C_j,C_{j-1})(C_j,C_{j+1})<0,
		&
		if $(C_j,C_j)<0$,
		\\
		(C_{j-1},C_{j+1})>0,
		&
		if $(C_j,C_j)=0$.
	\end{cases*}
\end{align}
To make formulas in this note more readable, we set in the following $\mathrm{sgn}(v,w)=\mathrm{sgn}((v,w))$ for $v,w\in V$. Then we put (following the notation in \cite{FunkeKudla2022}, up to a sign change)
\[
\mathbf{w}(x;\mathcal{C})
=
\sum_{j=1}^N
\mathrm{sgn}(x,C_j)\mathrm{sgn}(x,C_{j+1}),
\quad
x\in V.
\]
Here by convention, $C_{N+1}=C_1$. Take a negative vector $v\in V$ and set
\[
\Phi(x;\mathcal{C})
=
\mathbf{w}(x;\mathcal{C})
-
\mathbf{w}(v;\mathcal{C}).
\]
We then define
\[
\theta_{\mu}(\tau)=\theta_{\mu}(\tau;\mathcal{C})=
\sum_{x\in\mu+L}
\Phi(x;\mathcal{C})
q^{(x,x)/2}.
\]
The main result of this note is
\begin{theorem}\label{main theorem}
	\begin{enumerate}
		\item
		Suppose $\mathcal{C}$ satisfies (I.1)-(I.3), then $\mathbf{w}(x;\mathcal{C})$ is constant on $V_{<0}$ (the set of negative vectors in $V$) and the indefinite theta series $\theta_{\mu}(\tau)$ is termwise absolutely convergent.

		\item
		Suppose $\mathcal{C}$ satisfies (I.1)-(1.2) and the indefinite theta series $\theta_{\mu}(\tau)$ is termwise absolutely convergent. If moreover no three consecutive vectors in $\mathcal{C}$ are all null-vectors, then $\mathcal{C}$ also satisfies (I.3).
	\end{enumerate}	
\end{theorem}

The value $\mathbf{w}(v;\mathcal{C})$ has the following obvious interpretation: write $v^{\perp}$ for the orthogonal complement of $v$ (a hyperplane in $V$). Then
\begin{enumerate}
	\item
	$(v,C_j)(v,C_{j+1})\ge0$ if and only if the straight segment $\gamma_{j}$ in $V$ connecting $C_j$ to $C_{j+1}$ passes through the hyperplane $v^{\perp}$;

	\item
	$(v,C_j)(v,C_{j+1})<0$ if and only if this segment lies outside this hyperplane.
\end{enumerate}
Write $\gamma(\mathcal{C})$ for the sum $\gamma_{1}+\gamma_{2}+\cdots+\gamma_{N}$, a piecewise smooth loop in $V$. Then the value $\mathbf{w}(v;\mathcal{C})$ measures how many times this loop $\gamma(\mathcal{C})$ passes through the hyperplane $v^{\perp}$, assuming that $(v,C_j)\ne0$ for any $j=1,\cdots,N$. If $\gamma(\mathcal{C})$ passes through $r$ times through $v^{\perp}$, then one has $\mathbf{w}(v;\mathcal{C})=N-4r$.

The first part of the theorem is also proved in \cite{FunkeKudla2022} under the further condition that $(C_j,C_j)<0$ for all $j=1,\cdots,N$. As explained in \textit{loc.cit}, (I.3) means that the $2$-dimensional oriented plane $[C_j,C_{j+1}]$ spanned by $C_j$ and $C_{j+1}$ all lie in the same connected component of $D(V)$, the Grassmannian of oriented negative $2$-planes in $V$. This part corrects the conjecture in \cite[C]{AlexandrovBanerjeeManschotPioline2018b}, which does not include the term $\mathbf{w}(v;\mathcal{C})$ in the definition of $\Phi(x;\mathcal{C})$: as explained in \cite{FunkeKudla2022}, for $N=3$, $\mathbf{w}(v;\mathcal{C})$ is non-zero (we can also use the preceding interpretation of $\mathbf{w}(v;\mathcal{C})$) and therefore $\theta_{\mu}(\tau)$ can not be termwise absolutely convergent without this term.

The theta series $\theta_{\mu}$ defines a mock modular form of weight $(n+2)/2$, whose modular completion is given by
\[
\widehat{\theta}_{\mu}
(\tau)
=
\sum_{x\in\mu+L}
\left(
-\mathbf{w}(v;\mathcal{C})
+\sum_{j=1}^NE_2(C_j,C_{j+1};\sqrt{2}x)
\right)
q^{(x,x)/2},
\]
where
\[
E_2(C,C';x)
=
\int_z
e^{\pi(y-\mathrm{pr}_z(x),y-\mathrm{pr}_z(x))}
\mathrm{sgn}(y,C)
\mathrm{sgn}(y,C')
dy
\]
with $z$ the subspace of $V$ generated by $C$ and $C'$ and $\mathrm{pr}_z(x)$ the projection of $x$ to $z$. We refer the reader to \cite[C]{AlexandrovBanerjeeManschotPioline2018b} for more details.

We say a vector $v\in V$ is \textbf{regular at $\mathcal{C}$} if $(v,C_j)\ne0$ for any $j=1,\cdots,N$. Write $\mathrm{Reg}(\mathcal{C})$ for the subset of $V$ consisting of $v$ regular at $\mathcal{C}$, which is an open dense subset of $V$ and also a finite disjoint union of polyhedral cones.

The strategy of proof of \ref{main theorem}(1) is based on the following simple observation: the function $\mathbf{w}(x;\mathcal{C})$ is locally constant on $\mathrm{Reg}(\mathcal{C})$ and it is constant on $V_{<0}$. So we only need to take care of the connected components of $\mathrm{Reg}(\mathcal{C})$ whose closure intersect with the closure $\overline{V}_{<0}$ only at $0$. On such components, $(x,x)$ takes positive values and in fact is bounded below by some positive definite inner product $(x,x)_{\mathcal{C}}$.

\section{Convergence of indefinite theta series}
As in the introduction, let $(V,(-,-))$ be an indefinite inner product space over $\mathbb{R}$ of signature $(n,2)$. Take $N$ non-zero vectors $\mathcal{C}=\{C_1,\cdots,C_N\}$ in $V$ satisfying conditions (I.1)-(I.3).

\begin{lemma}\label{w is locally constant}
	The restriction of the map $\mathbf{w}(-;\mathcal{C})$ to $\mathrm{Reg}(\mathcal{C})$ is locally constant.
\end{lemma}
\begin{proof}
	This is because each $\mathrm{sgn}(x,C_j)$ is locally constant on $\mathrm{Reg}(\mathcal{C})$.
\end{proof}
This is also true without $\mathcal{C}$ satisfying (I.1)-(I.3).

\begin{figure}
	\begin{tikzpicture}
		\draw [->,very thick] (0,0)--(1,2);
		    \node[left] at (1,2) {$C_{j}$};
		\draw [->,very thick] (0,0)--(-2,0);
		    \node[below] at (-2,0) {$C_{j-1}$};
		\draw [->] (0,0)--(2,-2);
		    \node[below,very thick] at (2,-2) {$C_{j+1}$};
		\draw [dashed,red,very thick] (-2,-3)--(2,3);
		    \node[right] at (2,3) {\textcolor{red}{$v^{\perp}\cap V'$}};
		\draw [dashed,very thick] (-1.5,-3)--(2,4);
		    \node[right] at (-1.5,-3) {$\mathbb{R}C_j$};
		\draw [->,red,very thick] (0,0)--(3,-2);
		    \node[right] at (3,-2) {\textcolor{red}{$v$}};
		    \node[left] at (-2,2) {$V'$};
	\end{tikzpicture}
	\caption{}\label{Figure.1}
\end{figure}

The following three lemmas are the technical heart of this note. The main idea can be explained as follows (see Figure.\ref{Figure.1} for an illustration): suppose $C_{j-1},C_j,C_{j+1}$ lie in the same \textit{negative} plane $V'$ (this is possible under the conditions (I.1) and (I.2)), then (I.3) implies that $C_{j-1}$ and $C_{j+1}$ lie in different connected components (half planes) of $V'\backslash\mathbb{R}C_j$.
Therefore, for a negative vector $v$, when $(v,C_j)$ is very small (compared to the norms of $v,C_{j\pm1},C_j,(C_j,C_{j\pm1})$), the intersection $V'\cap v^{\perp}$ separates the two vectors $C_{j-1},C_{j+1}$. As a result, the value $\mathrm{sgn}(v,C_{j-1})\mathrm{sgn}(v,C_{j+1})$ is negative. This means that when we perturb $v$ slightly such that $(v,C_j)$ changes signs (assuming the signs $\mathrm{sgn}(v,C_i)$ remain unchanged for all $i\ne j$), the value $\mathbf{w}(v;\mathcal{C})$ is unchanged.

We fix a basis $(e_1,e_2,\cdots,e_{n+2})$ of $V$ such that the bilinear form $(-,-)$ is represented by the diagonal matrix $\mathrm{diag}(-1_2,1_n)$.
\begin{lemma}\label{v_1v_3<0}
	Fix $j=1,\cdots,N$.
	For any vector $v\in V_{<0}$ such that $(v,C_j)=0$, we have
	\[
	\mathrm{sgn}(v,C_{j-1})\mathrm{sgn}(v,C_{j+1})<0.
	\]
\end{lemma}
\begin{proof}
	First we assume that $C_j$ is not a null vector.  We can suppose that the vectors $C_j$ and $v$ are of the form: $C_j=(1,0,\cdots)$ and $v=(0,1,0,\cdots)$. Then we can write $C_{j-1}=(a,b,c,d,\cdots)$ and $C_{j+1}=(a',b',c',d',\cdots)$. (I.2) and (I.3) give
	\begin{align*}
		b^2-c^2-d^2-\cdots
		&
		\ge0,
		\\
		(b')^2-(c')^2-(d')^2-\cdots
		&
		\ge
		0,
		\\
		bb'-cc'+dd'+\cdots
		&
		\le
		0.
	\end{align*}
	From this, we get
	\[
	(v,C_{j-1})(v,C_{j+1})
	=
	bb'<0.
	\]

	Next assume $(C_j,C_j)=0$, then (I.2) implies
	\[
	(C_j,C_{j-1})=0=(C_j,C_{j+1}).
	\]
	As in the above, we can assume $C_j=(1,0,1,0,\cdots)$ and $v=(0,1,0,\cdots)$. Then we can write $C_{j-1}=(a,b,a,c,d,\cdots)$ and $C_{j+1}=(a',b',a',c',d',\cdots)$. (I.1) and (I.3) imply
	\begin{align*}
		-b^2+(c^2+d^2+...)
		&
		\le0
		\\
		-(b')^2+((c')^2+(d')^2+...)
		&
		\le0
		\\
		-bb'+(cc'+dd'+...)
		&
		>0.
	\end{align*}
	We deduce that $bb'<0$ and thus $(v,C_{j-1})(v,C_{j+1})=bb'<0$.
\end{proof}

 For two vectors $v,w\in V$, we write
 \[
 \gamma_{v,w}(s)=(1-s)v+sw,
 \quad
 s\in[0,1]
 \]
 for the path/segment in $V$ connecting $v$ to $w$. For each $j=1,\cdots,N$, we write
 \[
 \gamma_j(s)=\gamma_{C_j,C_{j+1}}(s).
 \]
 Then (I.1) and (I.2) imply that $\gamma_j(s)$ is a non-positive vector for any $s\in[0,1]$.
\begin{lemma}\label{w(gamma(t);C) is constant on t}
	Let $\gamma\colon[0,1]\rightarrow V_{<0}$ be a continuous map such that there are only finitely many $t\in[0,1]$ with $\gamma(t)\in V_{<0}\backslash\mathrm{Reg}(\mathcal{C})$. Then $\mathbf{w}(\gamma(t);\mathcal{C})$ is a constant map on $t\in[0,1]$.
\end{lemma}
\begin{proof}
	If $\gamma([0,1])\subset\mathrm{Reg}(\mathcal{C})$, then the conclusion holds thanks to Lemma \ref{w is locally constant}. We can then assume there is only one $t_0\in[0,1]$ such that $\gamma(t_0)\notin\mathrm{Reg}(\mathcal{C})$. We can put $t_0=1$ and $(\gamma(1),C_{j_1})=\cdots=(\gamma(1),C_{j_p})=0$ for indices $j_1<j_2<\cdots<j_p$ (we have $j_r+1\ne j_{r+1}$ for any $r=1,\cdots,p-1$ by Lemma \ref{v_1v_3<0}) and $(v,C_j)\ne0$ for other indices $j$. We can then assume that $\mathrm{sgn}(\gamma(t),C_{j_r})<0$ for any $t\in[0,1[$. By Lemma \ref{v_1v_3<0}, we have $\mathrm{sgn}(\gamma(t),C_{j_r-1})\mathrm{sgn}(\gamma(t),C_{j_r+1})<0$ for any $t\in[0,1]$. As a result,
	\[
	\mathrm{sgn}(\gamma(t),C_{j_r-1})\mathrm{sgn}(\gamma(t),C_{j_r})+\mathrm{sgn}(\gamma(t),C_{j_r})\mathrm{sgn}(\gamma(t),C_{j_r+1})
	=
	0.
	\]
	For $j\ne j_r-1,j_r$, the cardinal of the intersection $\gamma(t)^{\perp}\cap\gamma_j(]0,1[)$ is constant on $t\in[0,1]$ and therefore $\mathrm{sgn}(\gamma(t),C_j)\mathrm{sgn}(\gamma(t),C_{j+1})$ is constant on $t\in[0,1]$. So we conclude that $\mathbf{w}(\gamma(t);\mathcal{C})$ is indeed constant on $t\in[0,1]$.
\end{proof}

\begin{lemma}\label{w(v;C) is constant on negative vectors}
	For any negative vectors $v,w\in V_{<0}$, we have $\mathbf{w}(v;\mathcal{C})=\mathbf{w}(w;\mathcal{C})$.
\end{lemma}
\begin{proof}
	It suffices to construct a path $\gamma\colon[0,1]\rightarrow V_{<0}$ as in Lemma \ref{w(gamma(t);C) is constant on t} connecting $v$ and $w$. Since $\mathrm{Reg}(\mathcal{C})\cap V_{<0}$ is open dense in $V_{<0}$, we can find $v_1,w_1\in V_{<0}\cap\mathrm{Reg}(\mathcal{C})$ in sufficiently small open neighborhoods of $v$ and $w$ such that the paths $\gamma_{v,v_1}$ and $\gamma_{w_1,w}$ satisfy Lemma \ref{w(gamma(t);C) is constant on t}. Moreover we can always find a negative vector $u\in V_{<0}$ such that $\gamma_{v_1,u}$ and $\gamma_{u,w_1}$ both lie in $V_{<0}$ and satisfy Lemma \ref{w(gamma(t);C) is constant on t}: indeed, under a suitable basis of $V$ such that the quadratic form $(-,-)$ is represented by the matrix $\mathrm{diag}(-1_2,1_n)$, we can write $v_1$ and $w_1$ in coordinates as $v_1=(1,0,\cdots,0)$ and $w_1=(a_1,a_2,\cdots,a_{n+2})$ with $(w_1,w_1)=-1$. Then we can always find a vector $u=(x_1,x_2,0,\cdots,0)\in V$ such that $x_1^2+x_2^2=-(u,u)=1$, $x_1^2=(v_1,u)^2<1$ ($v_1$ and $u$ span a negative subspace of dimension $2$) and $(a_1x_1+a_2x_2)^2=(u,w_1)^2<1$ ($u$ and $w_1$ span a negative subspace of dimension $2$). Now it is easy to see that the path $\gamma_{v,v_1}+\gamma_{v_1,u}+\gamma_{u,w_1}+\gamma_{w_1,w}$ in $V_{<0}$ satisfies Lemma \ref{w(gamma(t);C) is constant on t}.
\end{proof}
We write $\mathbf{w}_{\mathcal{C}}
=
\mathbf{w}(v;\mathcal{C})$ for some $v\in V_{<0}$.

\begin{proof}[Proof of Theorem\ref{main theorem}(1)]
	The above proposition gives the first part of Theorem \ref{main theorem}(1). For the convergence, we proceed by a case-by-case study of the vectors in $V$, the main idea is that, as long as a vector $x$ lies in (the boundary of) a connected component of $\mathrm{Reg}(\mathcal{C})$ which has non-empty intersection with $V_{<0}$, we can approximate $x$ with a vector $x'$ in this component and these two values $\mathbf{w}(x;\mathcal{C})$ and $\mathbf{w}(x';\mathcal{C})$ should be equal.
	
	It is easy to see that any non-zero vector $x$ in $V$ belongs to one and only one of the following cases:
	\begin{enumerate}
		\item
		$x\in U$ for some connected component $U$ (a polyhedral cone) of $\mathrm{Reg}(\mathcal{C})$ such that $U\cap V_{<0}=\emptyset$. Since $\mathbf{w}(x;\mathcal{C})$ is constant on $U$ and also on $V_{<0}$, so $\mathbf{w}(x;\mathcal{C})=\mathbf{w}_{\mathcal{C}}$.

		\item
		$x$ does not satisfy the preceding case but lies in $\overline{U}_1\cap\overline{U}_2$ for two distinct connected components $U_1,U_2$ of $\mathrm{Reg}(\mathcal{C})$ such that $U_1\cap V_{<0},U_2\cap V_{<0}\neq\emptyset$. We can assume that $\overline{U}_1\cap\overline{U}_2$ is contained in $C_{j_i}^{\perp}$ for $j_1<j_2<\cdots<j_k$ and not in any other $C_i^{\perp}$. Then these vectors $C_{j_1},\cdots,C_{j_k}$ are all parallel and it is easy to see $j_{i+1}-j_i>1$ for any $i$. Note that $\overline{U}_1\cap\overline{U}_2$ is a polyhedral cone in a subspace of $V$ of co-dimension $1$. First we assume that $x$ lies in the interior of the polyhedral cone $\overline{U}_1\cap\overline{U}_2$. So we can choose two vectors $x_1\in U_1$ and $x_2\in U_2$ in a sufficiently small open neighborhood of $x$ such that
		\begin{align*}
			\mathrm{sgn}(x_1,C_{j_i})=-\mathrm{sgn}(x_2,C_{j_i})
			&
			=1,\,
			\quad
			\forall
			i=1,\cdots,k
			\\
			\mathrm{sgn}(x_1,C_j)=\mathrm{sgn}(x_2,C_j)
			&
			=
			\mathrm{sgn}(x,C_j),\,
			\quad
			\forall
			j\notin\{j_1,\cdots,j_k\}.
		\end{align*}
		By definition of $\mathbf{w}(x_1;\mathcal{C})$ and $\mathbf{w}(x_2;\mathcal{C})$, one gets
		\begin{align*}
			0
			&
			=
			\mathbf{w}(x_1;\mathcal{C})
			-
			\mathbf{w}(x_2;\mathcal{C})
			\\
			&
			=
			\sum_{i=1}^k
			\mathrm{sgn}(x_1,C_{j_i-1})\mathrm{sgn}(x_1,C_{j_i})
			+
			\mathrm{sgn}(x_1,C_{j_i+1})\mathrm{sgn}(x_1,C_{j_i})
			\\
			&
			\quad-
			\mathrm{sgn}(x_2,C_{j_i-1})\mathrm{sgn}(x_2,C_{j_i})
			-
			\mathrm{sgn}(x_2,C_{j_i+1})\mathrm{sgn}(x_2,C_{j_i})
			\\
			&
			=
			2\sum_{i=1}^k
			\mathrm{sgn}(x_1,C_{j_i-1})\mathrm{sgn}(x_1,C_{j_i})
			+
			\mathrm{sgn}(x_1,C_{j_i+1})\mathrm{sgn}(x_1,C_{j_i})
		\end{align*}
		On the other hand,
		\begin{align*}
			&
			\mathbf{w}(x_1;\mathcal{C})
			-
			\mathbf{w}(x;\mathcal{C})
			\\
			=
			&
			\sum_{i=1}^k
			\mathrm{sgn}(x_1,C_{j_i-1})\mathrm{sgn}(x_1,C_{j_i})
			+
			\mathrm{sgn}(x_1,C_{j_i+1})\mathrm{sgn}(x_1,C_{j_i})
		\end{align*}
		Thus $\mathbf{w}(x;\mathcal{C})=\mathbf{w}_{\mathcal{C}}$.

		If $x$ lies in the boundary of the polyhedral cone $\overline{U}_1\cap\overline{U}_2$, we can repeat the above argument by choosing points $x_1',x_2'$ from $\overline{U}_1\backslash U_1,\overline{U}_2\backslash U_2$ (we have already shown that $\mathbf{w}(x_1';\mathcal{C})=\mathbf{w}(x_2';\mathcal{C})=\mathbf{w}_{\mathcal{C}}$). This process will continue if $x$ lies in the boundary of the boundary of $\overline{U}_1\cap\overline{U}_2$, etc. It will terminate in finitely many steps as there are only finitely many connected components of $\mathrm{Reg}(\mathcal{C})$.

		\item
		$x$ does not satisfy the preceding cases but lies in a connected component $U$ of $\mathrm{Reg}(\mathcal{C})$ such that $U\cap V_{<0}=\emptyset$ and $\overline{U}\cap\overline{V}_{<0}\ne\emptyset$. Then there is another connected component $U'$ of $\mathrm{Reg}(\mathcal{C})$ such that $U'\cap V_{<0}\ne\emptyset$ and $\overline{U}\cap\overline{U'}\cap\overline{V}_{<0}\ne\emptyset$. Take $x_0\in\overline{U}\cap\overline{U'}\cap\overline{V}_{<0}$ and $x'\in U'$ and assume that $x,x'$ are in a sufficiently small open neighborhood of $x_0$ (note that $\mathbf{w}(x;\mathcal{C})$ is constant on $U$). Suppose that $\overline{U}\cap\overline{U'}$ is contained in $C_j^{\perp}$ for some $j$. There are two cases to consider: $(C_j,C_j)$ is null or not.

		If $(C_j,C_j)<0$, we can write
		\begin{align*}
			C_j
			&
			=(0,1,0,\cdots),
			&x_0
			=(1,0,1,0,\cdots),
			\\
			C_{j-1}
			&
			=(a,b,c,d,\cdots),
			&C_{j+1}
			=(a',b',c',d',\cdots).
		\end{align*}
	    Conditions (I.2) and (I.3) give
		\begin{align*}
			a^2-c^2-d^2-\cdots
			&
			\ge 0,
			\\
			(a')^2-(c')^2-(d')^2-\cdots
			&
			\ge
			0,
			\\
			aa'-cc'-dd'-\cdots
			&
			<0.
		\end{align*}
		In particular, we have $aa'<0$ and thus $(x_0,C_{j-1})(x_0,C_{j+1})=(-a+c)(-a'+c')<0$. So $\mathrm{sgn}(x_0,C_{j-1})+\mathrm{sgn}(x,C_{j+1})=0$.

		If $(C_j,C_j)=0$ and $x_0$ is not parallel to $C_j$, then we proceed as above and get $\mathrm{sgn}(x_0,C_{j-1})+\mathrm{sgn}(x_0,C_{j+1})=0$.

		If $(C_j,C_j)=0$ and $x_0$ is parallel to $C_j$, then (I.2) gives
		\[
		\mathrm{sgn}(x,C_{j-1})+\mathrm{sgn}(x,C_{j+1})=0+0=0.
		\]
		Since $x,x'$ and $x_0$ are all sufficiently near to each other, we can use the argument in (2) to conclude that $\mathbf{w}(x;\mathcal{C})=\mathbf{w}(x_0;\mathcal{C})=\mathbf{w}(x';\mathcal{C})=\mathbf{w}_{\mathcal{C}}$.

		\item
		$x$ does not satisfy the preceding cases but lies in $\overline{U}_1\cap\overline{U}_2$ for two distinct connected components $U_1,U_2$ of $\mathrm{Reg}(\mathcal{C})$ such that $U_1\cap V_{<0}\neq\emptyset$ and $U_2\cap V_{<0}=\emptyset$. Note that we have $\overline{U}_2\cap\overline{V}_{<0}\neq\emptyset$. From (3), we know that on $U_1$ and $U_2$, $\mathbf{w}(-;\mathcal{C})$ is equal to $\mathbf{w}_{\mathcal{C}}$. Now we use the argument in (2) to conclude that $\mathbf{w}(x;\mathcal{C})=\mathbf{w}_{\mathcal{C}}$.

		\item
		$x$ does not satisfy the preceding cases but lies in $U$ for some connected component $U$ of $\mathrm{Reg}(\mathcal{C})$ such that $\overline{U}\cap\overline{V}_{<0}=\{0\}$. In this case we fix a positive definite inner product $(-,-)_{\mathcal{C}}$ on $V$ and consider the quotient $(x,x)/(x,x)_{\mathcal{C}}$, which is a continuous function on $\overline{U}\backslash\{0\}$ and moreover it is homogeneous of degree $0$. Since it is everywhere positive on $\overline{U}\backslash\{0\}$, it admits a positive infimum $r_{U,\mathrm{inf}}$, that is, for any $x\in\overline{U}\backslash\{0\}$,
		\[
		(x,x)\ge r_{U,\mathrm{inf}}(x,x)_{\mathcal{C}}.
		\]
	\end{enumerate}
    We write $\mathcal{U}$ for the set of $U$ as in (5) above, fix the inner product $(-,-)_{\mathcal{C}}$ for all $U\in\mathcal{U}$ and set $r_{\mathrm{inf}}=\min_{U\in\mathcal{U}}r_{U,\mathrm{inf}}>0$. We deduce immediately that if $x$ is not in $U$ for any $U\in\mathcal{U}$, then $\mathbf{w}(x;\mathcal{C})=\mathbf{w}_{\mathcal{C}}$; if $x$ is in $U$ for some $U\in\mathcal{U}$, then $(x,x)\ge r_{\mathrm{inf}}(x,x)_{\mathcal{C}}$. From this one get the absolute convergence of $\theta_{\mu}$:
    \begin{align*}
    	|\sum_{x\in\mu+L}\Phi(x;\mathcal{C})q^{(x,x)/2}|
    	&
    	\le
    	2N
    	\sum_{x\in U\text{for some }U\in\mathcal{C}}
    	e^{-\pi vr_{\mathrm{inf}}(x,x)_{\mathcal{C}}/2}
    	\\
    	&
    	\le
    	2N
    	\sum_{x\in\mu+L}
    	e^{-\pi vr_{\mathrm{inf}}(x,x)_{\mathcal{C}}/2}.
    \end{align*}
\end{proof}

We have shown that for any set $\mathcal{C}$ of $N$ non-zero vectors satisfying (I.1)-(I.3), the map $\mathbf{w}(v;\mathcal{C})$ is constant on $v\in V_{<0}$. In fact, we have the following partial converse
\begin{proposition}
	For any set of $N$ non-zero vectors $\mathcal{C}=\{C_1,\cdots,C_N\}$ in $V$ satisfying (I.1) and (I.2), if $\mathbf{w}(v;\mathcal{C})$ is constant on $v\in V_{<0}$, then $\mathcal{C}$ also satisfies (I.3).
\end{proposition}
\begin{proof}
	Take $v\in V_{<0}\cap\mathrm{Reg}(\mathcal{C})$. We can choose $N$ non-zero vectors $\mathcal{C}'=\{C_1',\cdots,C_N'\}$ with $C_i'$ in a sufficiently small neighborhood of $C_i$ for all $i=1,\cdots,N$ such that the following conditions are satisfied
	\begin{enumerate}
		\item 
		$\mathbf{w}(v;\mathcal{C})=\mathbf{w}(v;\mathcal{C}')$,

		\item
		the vectors $C_j'$ are not parallel to each other,

		\item
		$\mathcal{C}$ satisfies (I.1)-(I.3) if and only if $\mathcal{C}'$ satisfies (I.1)-(I.3).
	\end{enumerate}
    Therefore we can assume from the onset that the vectors $C_j$ are not parallel to each other.

	Fix an index $j=1,\cdots,N$. Take a negative vector $v\in V$ such that $(v,C_j)=0$ and $(v,C_i)\ne0$ for any $i\ne j$ (this is always possible since the $C_j$ are not parallel to each other). Then we can find another negative vector $w\in V$ in a sufficiently small open neighborhood of $v$ such that for any $s\in[0,1]$, the following conditions hold
	\begin{enumerate}
		\item
		the vector $\gamma_{w,v}(s)$ is a negative vector for any $s\in[0,1]$;

		\item
		$(\gamma_{w,v}(s),C_i)\ne0$ unless $i=j$ and $s=1$;

		\item
		for any $i\ne j-1,j$, $v^{\perp}\cap\gamma_i(]0,1[)=\emptyset$ if and only if $\gamma_{w,v}(s)^{\perp}\cap\gamma_i(]0,1[)=\emptyset$ for any $s\in[0,1]$;

		\item
		$\mathrm{sgn}(v,C_i)\mathrm{sgn}(\gamma_{w,v}(s),C_i)=1$ for $i=j-1,j+1$ and for any $s\in[0,1]$.
	\end{enumerate}
    Thus $\mathrm{sgn}(\gamma_{w,v}(s),C_i)\mathrm{sgn}(\gamma_{w,v}(s),C_{i+1})$ are all constant on $s\in[0,1]$ (for any $i\ne j-1,j$), we deduce that
    \[
    \mathrm{sgn}(\gamma_{w,v}(s),C_{j-1})\mathrm{sgn}(\gamma_{w,v}(s),C_j)+\mathrm{sgn}(\gamma_{w,v}(s),C_j)\mathrm{sgn}(\gamma_{w,v}(s),C_{j+1})
    \]
    is constant on $s\in[0,1]$. Setting $s=1$, we know that this map is equal to $0$. Now take $s\ne1$ and we must have $\mathrm{sgn}(\gamma_{w,v}(s),C_{j-1})\mathrm{sgn}(\gamma_{w,v}(s),C_{j+1})=-1$ and therefore $(v,C_{j-1})(v,C_{j+1})<0$.
    
    If $(C_j,C_j)<0$, we can write
    \begin{align*}
    	v=(1,0,\cdots),
    	&
    	\quad
    	C_j=(0,1,0,\cdots),
    	\\
    	C_{j-1}=(a,b,c,d,\cdots),
    	&
    	\quad
    	C_{j+1}=(a',b',c',d',\cdots).
    \end{align*}
    (I.2) gives
    \[
    a^2-c^2-d^2-\cdots>0,
    \quad
    (a')^2-(c')^2-(d')^2-\cdots>0.
    \]
    Since $(v,C_{j-1})(v,C_{j+1})=aa'<0$, we have $aa'-cc'-dd'-\cdots<0$ and thus (I.3) holds for $j$.

    If $(C_j,C_j)=0$, we can write
    \begin{align*}
    	v=(1,0,\cdots),
    	&
    	\quad
    	C_j=(0,1,1,0,\cdots),
    	\\
    	C_{j-1}=(a,b,b,0,\cdots)+D,
    	&
    	\quad
    	C_{j+1}=(a',b',b',0,\cdots)+D'
    \end{align*}
    where $D,D'\in\mathbb{R}(e_3,\cdots,e_{n+2})$. (I.1) gives
    \[
    a^2-(D,D)\ge0,
    \quad
    (a')^2-(D',D')\ge0.
    \]
    Combined with $(v,C_{j-1})(v,C_{j+1})=aa'<0$, we get that $(C_{j-1},C_{j+1})=-aa'+(D,D')\ge0$, and the equality holds if and only if $a^2=(D,D)$, $(a')^2=(D',D')$ and $D$ is parallel to $D'$. In particular, if the equality holds, we have $(C_{j-1},C_{j-1})=-a^2+(D,D)=0$ and similarly $(C_{j+1},C_{j+1})=0$, contradicting the assumption that no three consecutive vectors in $\mathcal{C}$ are all null vectors. Therefore we must have $(C_{j-1},C_{j+1})>0$.
\end{proof}

\begin{proof}[Proof of Theorem \ref{main theorem}(2)]
	The absolute convergence of $\theta_{\mu}(\tau)$ implies that the discrete-valued map $\Phi(x;\mathcal{C})$ is identically zero if $(x,x)<0$. Then we apply the above proposition.
\end{proof}

\begin{remark}
	The strategy of proof of \ref{main theorem} can be easily adapted to the case of indefinite theta series $\theta_{\mu}(\tau;\{C_1,C_2\})$ for an integral lattice $(L,(-,-))$ of signature $(n,1)$ mentioned in the introduction. In particular, for two negative vectors $C_1,C_2$ in $V$, if $\theta_{\mu}(\tau;\{C_1,C_2\})$ are termwise absolutely convergent, then $(C_1,C_2)<0$.
\end{remark}

\end{document}